\pdfoutput=1 
\documentclass[a4paper]{article}
\usepackage{amsthm,amssymb,amsmath,enumerate,graphicx}

\newtheorem{THM}{Theorem}[section]
\newtheorem{LEM}[THM]{Lemma}

\def\shift(#1)(#2){\!\!\downarrow\!{}^{#1}_{\raise .1ex\vbox to 0pt{\vss\hbox{$\scriptstyle #2$}}}\,}
\def\ucl(#1){\lfloor #1 \rfloor}
\def\dcl(#1){\lceil #1 \rceil}
\def\specrel#1#2{\mathrel{\mathop{\kern0pt #1}\limits_{#2}}}

\def\F{\mathcal F}

\renewcommand\O{\mathcal O}
\renewcommand\P{\mathcal P}

\def\lowfwd #1#2#3{{\mathop{\kern0pt #1}\limits^{\kern#2pt\raise.#3ex
\vbox to 0pt{\hbox{$\scriptscriptstyle\rightarrow$}\vss}}}}
\def\lowbkwd #1#2#3{{\mathop{\kern0pt #1}\limits^{\kern#2pt\raise.#3ex
\vbox to 0pt{\hbox{$\scriptscriptstyle\leftarrow$}\vss}}}}
\def\fwd #1#2{{\lowfwd{#1}{#2}{15}}}
\def\ve{\kern-1.5pt\lowfwd e{1.5}2\kern-1pt}
\def\vedash{{\mathop{\kern0pt e\lower.5pt\hbox{${}
     \scriptstyle'$}}\limits^{\kern0pt\raise.02ex
     \vbox to 0pt{\hbox{$\scriptscriptstyle\rightarrow$}\vss}}}}
\def\ev{\kern-1pt\lowbkwd e{0.5}2\kern-1pt}
\def\vf{\kern-2pt\lowfwd f{2.5}2\kern-1pt}
\def\vfdash{{\mathop{\kern0pt f\raise 1pt\hbox{${}
     \scriptstyle'$}}\limits^{\kern2pt\raise.02ex
     \vbox to 0pt{\hbox{$\scriptscriptstyle\rightarrow$}\vss}}}}

\def\vL{{\vec L}} 

\def\vr{\lowfwd r{1.5}2}
\def\rv{\lowbkwd r02}

\def\vrdash{{\mathop{\kern0pt r\lower.5pt\hbox{${}
     \scriptstyle'$}}\limits^{\kern0pt\raise.02ex
     \vbox to 0pt{\hbox{$\scriptscriptstyle\rightarrow$}\vss}}}}
\def\rvdash{{\mathop{\kern0pt r\lower.5pt\hbox{${}
     \scriptstyle'$}}\limits^{\kern0pt\raise.02ex
     \vbox to 0pt{\hbox{$\scriptscriptstyle\leftarrow$}\vss}}}}

\def\vR{{\vec R}} 

\def\vs{\lowfwd s{1.5}1}
\def\sv{\lowbkwd s{1.5}1}
\def\vsdash{{\mathop{\kern0pt s\lower.5pt\hbox{${}
     \scriptstyle'$}}\limits^{\kern0pt\raise.02ex
     \vbox to 0pt{\hbox{$\scriptscriptstyle\rightarrow$}\vss}}}}
\def\svdash{{\mathop{\kern0pt s\lower.5pt\hbox{${}
     \scriptstyle'$}}\limits^{\kern0pt\raise.02ex
     \vbox to 0pt{\hbox{$\scriptscriptstyle\leftarrow$}\vss}}}}
\def\vsddash{{\mathop{\kern0pt s\lower.5pt\hbox{${}
     \scriptstyle''$}}\limits^{\kern0pt\raise.02ex
     \vbox to 0pt{\hbox{$\scriptscriptstyle\rightarrow$}\vss}}}}
\def\vso{\lowfwd {s_0}11}
\def\svo{\lowbkwd {s_0}02}

\def\vsidash{{\mathop{\kern0pt s_i\kern-3.5pt\lower.3pt\hbox{${}
     \scriptstyle'$}}\limits^{\kern0pt\raise.02ex
     \vbox to 0pt{\hbox{$\scriptscriptstyle\rightarrow$}\vss}}}}

\def\vS{{\hskip-1pt{\fwd S3}\hskip-1pt}} 

\def\vSr{{\vec S}_{\raise.1ex\vbox to 0pt{\vss\hbox{$\scriptstyle\ge\vr$}}}}
\def\vSstar{{\mathop{\kern0pt S\lower-1pt\hbox{$^*$}}\limits^{\kern2pt
     \vbox to 0pt{\hbox{$\scriptscriptstyle\rightarrow$}\vss}}}}
\def\vSdash{{\mathop{\kern0pt S\lower-1pt\hbox{${}
     \scriptstyle'$}}\limits^{\kern2pt\raise.1ex
     \vbox to 0pt{\hbox{$\scriptscriptstyle\rightarrow$}\vss}}}}

\def\vT{{\fwd T1}}
\def\vU{{\vec U}} 

\def\es{\emptyset}
\def\sub{\subseteq}
\def\supe{\supseteq}
\def\sm{\smallsetminus}
\def\td{tree-decom\-po\-si\-tion}

\newcommand\COMMENT[1]{}
\def\?#1{\vadjust{\vbox to 0pt{\vss\vskip-8pt\leftline{%
     \llap{\hbox{\vbox{\pretolerance=-1
     \doublehyphendemerits=0\finalhyphendemerits=0
     \hsize16truemm\tolerance=10000\small
     \lineskip=0pt\lineskiplimit=0pt
     \rightskip=0pt plus16truemm\baselineskip8pt\noindent
     \hskip0pt        
     #1\endgraf}\hskip7truemm}}}\vss}}}

\title{Tangles and the Mona Lisa}
 \author{Reinhard Diestel\and Geoffrey Whittle}
 \date{}

\begin{document}
\abovedisplayshortskip=-3pt plus3pt
\belowdisplayshortskip=6pt

\maketitle

\begin{abstract}\noindent
  We show how an image can, in principle, be described by the tangles of the graph of its pixels.

The tangle-tree theorem provides a nested set of separations that efficiently distinguish all the distinguishable tangles in a graph. This translates to a  small data set from which the image can be reconstructed.

The tangle duality theorem says that a graph either has a certain-order tangle or a tree structure witnessing that this cannot exist. This tells us the maximum resolution at which the image contains meaningful information.
   \end{abstract}

\section{Introduction}\label{sec:intro}

Image recognition and tangles in a graph have a common theme: they seek to identify regions of high coherence, and to separate these from each other. In both cases, the regions are fuzzy and therefore hard to nail down explicitly.

Let $K$ be a complete subgraph of order~$n$ in some graph~$G$. For every separation $\{A,B\}$ of~$G$ of order $k<n/2$ (say), the entire subgraph~$K$ lies squarely on one side or the other: either $A$ or $B$ contains all its vertices. If it lies in~$B$, say, we can {\it orient\/} the separation $\{A,B\}$ as~$(A,B)$ to indicate this.

Now let $H$ be an $n\times n$ grid in the same graph~$G$. Given a separation $\{A,B\}$ of order $k < n/2$, neither of the sides $A,B$ must contain all of~$H$. But one of them will contain more than 9/10 of the vertices of~$H$. If $B$ is that side, we can therefore still orient $\{A,B\}$ as $(A,B)$ to encode this information, even though there is a certain fuzziness about the fact that $H$ lives only essentially in~$B$, not entirely.

An orientation of all the separations of~$G$ of order less than some given integer~$k$ that $H$ induces in this way is called a {\em tangle of order~$k$\/}, or {\em $k$-tangle\/} in~$G$.%
   \footnote{Precise definitions will be given in Section~\ref{sec:defs}.}
   A graph can have many distinct tangles of a given order~-- e.g., as induced by different complete subgraphs or grids.

Tangles were first introduced by Robertson and Seymour~\cite{GMX}. The shift of paradigm that they have brought to the study of connectivity in graphs is that while large cohesive objects such as grids orient all the low-order separations of~$G$ in a consistent way (namely, towards them), it is meaningful to consider consistent orientations of all the low-order separations of~$G$ as `objects' in their own right: one lays down some axioms saying what `consistent' shall mean, and then various sensible notions of consistency give rise to different types of tangle, each a way of `consistently' orienting all the low-order separations, but in slightly different senses of consistency.

If one tangle orients a given separation $\{A,B\}$ as $(A,B)$ and another tangle orients it as $(B,A)$, we say that $\{A,B\}$ {\em distinguishes\/} these tangles. Distinct tangles of the same order~$k$ are always distinguished by some separation: otherwise they would orient all the separations of order~$<k$ in the same way, and thus be identical as tangles.

Since a given tangle, by definition, orients {\em every\/} separation of~$G$ of order~$<k$ for some~$k$, it is a complicated object that is not cheap to store, let alone easy to find. The set of {\em all\/} the tangles of a given graph, therefore, comes as a formidable data set: there can be exponentially many separations of given order,%
   \COMMENT{}
   and this order can take all values between 0 and the order of the graph.

The {\em tangle-tree\/} theorem~\cite{ProfilesNew,GMX}, however, asserts that a linear number of separations of~$G$~-- in fact, no more than~$|G|$~-- suffices to distinguish all the distinguishable tangles in~$G$: whenever two tangles of~$G$ differ at all (i.e., on some separation), they also differ on a separation in this small set. Moreover, the separations in this set are nested, and hence squeeze $G$ into a tree-structure: in the terminology of graph minors, there exists a \td\ of~$G$ such that every maximal tangle `lives in' a node of this tree, and the tree is minimal in that each of its nodes is home to a maximal tangle.

The main features of an image~-- such as the face in a portrait~-- are not unlike tangles: they are coherent regions which, given any way of cutting the picture in two along natural lines, would lie on one side or the other. Less blatant but still natural cutting lines might refine this division into regions, delineating a nose or an eye inside the face, for example. If we think of cutting lines as separations, whose order is the lower the more natural the line is to cut along, then the large and obvious features of the image (such as the faces in it) would orient all the low-order separations (towards them), while the more refined features such as eyes and noses would orient also some higher-order separations towards them. Thus, all it needs in order to capture the natural regions of an image as tangles is to define the order of separations in such a way that natural cutting lines have low order. We shall do this in Section~\ref{sec:apps}.

What does the tangle-tree theorem offer when applied to this scenario? The nested set of separations which it provides translates to a set of cutting lines that do not cross each other. In other words, it finds exactly the right lines that we need in order to divide the picture into smaller and smaller regions without cutting through coherent features prematurely. The nodes of the decomposition tree will correspond to the regions of the image that are left when all the lines are drawn. Just as the tangles can be identified from just the nested set of separations provided by the tangle-tree theorem, the distinctive regions~-- the features~-- of the image can be recovered from these carefully chosen cutting lines.\looseness=-1

These regions are likely to differ in importance: some are carved out early, by particularly natural lines, and would correspond to a tree node that is home to a tangle that has low order but is nevertheless maximal (is not induced by a higher-order tangle). Others are smaller, and appear only as the refining feature (such as an eye) of what was a region cut out by some obvious early lines (such as a face). For example, a portrait showing a face before a black backdrop would, at a modest resolution, capture the backdrop as a low-order tangle that does not get refined further, while the face is another low-order tangle that will be refined by several higher-order tangles (one each for an eye, the nose, an ear etc), which in turn may get refined further: maybe just one or two levels further in the case of the nose, but several levels for the eyes, so as to capture the details of the iris if these are clearly distinct.)

There is another theorem about tangles that can throw a light on the features of an image: the tangle duality theorem~\cite{TangleTreeAbstract,GMX}.  This theorem tells us something about the structure of the graph if it has {\em no\/} tangle of some given order~$k$: it then has an overall tree structure, a {\em\td\/} into small parts. In our application, this tree structure would once more come as a set of non-crossing cutting lines. But this time the regions left by these lines are small: too small to count as features. 

For example, a picture showing a pile of oranges might have many tangles up to some order~$k$, one for each orange, which are carved out by cutting lines around those oranges. But as soon as we allow less obvious cutting lines than these, we would admit lines cutting right through an orange. Since a single orange does not have interesting internal features, these lines would criss-cross through the oranges indiscriminately, without any focus giving them consistency. In tangle terms, this would be because there simply is no tangle of higher order inducing (i.e., refining) the tangle of an orange.%
   \COMMENT{}
   The tangle duality theorem would then point this out by cutting up the entire image into single pixels, in a tree-like way. In the language of complexity theory: the \td\ which the theorem assures us to exist if there is no high-order tangle other than those focusing on a single pixel will be an efficiently checkable witness to the fact there {\em exists no\/} such tangle: that we were not just not clever enough to spot it.\looseness=-1

This paper is organised as follows. In Section~\ref{sec:defs} we give a reasonably self-contained introduction to the theory of separation systems and abstract tangles in discrete structures (not just graphs), which leads up to formally precise statements of the versions of the tangle-tree theorem and the tangle duality theorem that we shall wish to apply. In Section~\ref{sec:apps} we translate the scenario of image recognition and image compression into the language of abstract separation systems and tangles, and apply the two theorems.

We wish to stress that we see this paper in no way as a serious contribution to any practical image recognition or compression problem. Rather, it is our hope to inspire those that know more about these things than we do, by offering some perhaps unexpected and novel ideas from the world of tangles. Its character, therefore, is deliberately one of `proof of concept'. In particular, the correspondence that we shall describe between the features of an image and the tangles in a suitable abstract separation system is the simplest possible one that we could think of: it works, but it can doubtless be improved. But this we leave to more competent hands.

\section{Abstract separation systems and tangles}\label{sec:defs}

Our aim in this section is to give a brief introduction to so-called abstract separation systems and their tangles, just enough to state precisely the two theorems we wish to apply to images recognition: the tangle-tree theorem and the tangle duality theorem. In the interest of brevity, we describe the technical setup more or less directly, without much additional motivation of why things are defined the way they are. Readers interested in this will find such discussions in~\cite{CDHH13CanonicalAlg, TreeSets, TangleTreeAbstract, ProfilesNew}.

\subsection{Separations of sets}\label{subsec:basics}

Given an arbitrary set~$V\!$, a {\em separation\/} of $V$ is a set $\{A,B\}$ of two subsets $A,B$ such that $A\cup B = V\!$. Every such separation $\{A,B\}$ has two \emph{orientations}: $(A,B)$ and $(B,A)$. 
Inverting these is an involution $(A,B) \mapsto (B,A)$ on the set of these \emph{oriented separations} of~$V\!$.

The oriented separations of a set are partially ordered as
\begin{equation*}
(A,B) \leq (C,D) :\Leftrightarrow A \subseteq C\text{ and } B\supseteq D.
\end{equation*}
Our earlier involution reverses this ordering:
\begin{equation*}
(A,B)\leq (C,D) \Leftrightarrow (B,A)\ge (D,C).
\end{equation*}

The oriented separations of a set form a lattice under this partial ordering, in which $(A\cap C, B\cup D)$ is the infimum of $(A,B)$ and~$(C,D)$, and $(A\cup C, B\cap D)$ is their supremum. The infima and suprema of two separations of a graph or matroid are again separations of that graph or matroid, so these too form lattices under~$\le$.

Their induced posets of all the oriented separations of order~$<k$ for some fixed~$k$, however, need not form such a lattice: when $(A,B)$ and $(C,D)$ have order~$<k$, this need not be the case for $(A\cap C, B\cup D)$ and~$(A\cup C, B\cap D)$.

\subsection{Abstract separation systems}\label{subsec:SeparationSystems}

A {\em separation system\/} $(\vS,\le\,,\!{}^*)$ is a partially ordered set $\vS$ with an order-reversing involution${\,}^*$.

The elements of a separation system~$\vS$ are called {\em oriented separations\/}. When a given element of $\vS$ is denoted as~$\vs$, its {\em inverse\/}~$\vs^*$ will be denoted as~$\sv$, and vice versa. The assumption that${\,}^*$ be {\em order-reversing\/} means that, for all $\vr,\vs\in\vS$,
\begin{equation}\label{invcomp}
\vr\le\vs\ \Leftrightarrow\ \rv\ge\sv.
\end{equation}
For subsets $R\sub\vS$ we write $R^*:=\{\,\rv\mid \vr\in R\,\}$.

An (unoriented) {\em separation\/} is a set of the form $\{\vs,\sv\}$, and then denoted by~$s$.%
   \footnote{To smooth the flow of the narrative we usually also refer to oriented separations simply as `separations' if the context or use of the arrow notation~$\vs$ shows that they are oriented.}
   We call $\vs$ and~$\sv$ the {\em orientations\/} of~$s$. If $\vs=\sv$, we call $s$ and~$\vs$ {\em degenerate.}

When a separation is introduced ahead of its elements and denoted by a single letter~$s$, its elements will then be denoted as $\vs$ and~$\sv$.
   Given a set $R\sub S$ of separations, we write $\vR := \bigcup R\sub\vS$%
   \COMMENT{}
   for the set of all the orientations of its elements. With the ordering and involution induced from~$\vS$, this is again a separation system.%
   \COMMENT{}

A separation $\vr\in\vS$ is {\em trivial in~$\vS$\/}, and $\rv$ is {\em co-trivial\/}, if there exists $s \in S$ such that $\vr < \vs$ as well as $\vr < \sv$. Note that if $\vr$ is trivial in~$\vS$ then so is every $\vrdash \le \vr$. An unoriented separation with no trivial orientation is {\em nontrivial\/}.
A separation $\vs$ is {\em small\/} if $\vs\le\sv$.  An unoriented separation is {\em proper\/} if it has no small orientation.

If there are binary operations $\vee$ and~$\wedge$ on a separation system $(\vU,\le\,,\!{}^*)$ such that $\vr\vee\vs$ is the supremum and $\vr\wedge\vs$ the infimum of $\vr$ and~$\vs$ in~$\vU$, we call $(\vU,\le\,,\!{}^*,\vee,\wedge)$ a {\em universe\/} of (oriented) separations. By~\eqref{invcomp}, it satisfies De~Mor\-gan's law:
\begin{equation}\label{deMorgan}
   (\vr\vee\vs)^* =\> \rv\wedge\sv.
\end{equation}%
   \COMMENT{}
    The universe~$\vU$ is {\em submodular\/} if it comes with a submodular \emph{order function}, a real function $\vs\mapsto |\vs|$ on~$\vU$%
   \COMMENT{}
   that satisfies $0\le |\vs| = |\sv|$ and 
 $$|\vr\vee\vs| + |\vr\wedge\vs|\le |\vr|+|\vs|$$
for all $\vr,\vs\in \vU$. We call $|s| := |\vs|$ the \emph{order} of $s$ and of~$\vs$. For every integer~$k>0$, then,
 $$\vS_k := \{\,\vs\in \vU : |\vs| < k\,\}$$
is a separation system~-- but not necessarily a universe, because $\vS_k$ is not normally closed under $\land$ and $\lor$.

Separations of a set~$V\!$, and their orientations, are clearly an instance of this abstract setup if we identify $\{A,B\}$ with $\{(A,B),(B,A)\}$. The small separations of~$V$ are those of the form~$(A,V)$, the proper ones those of the form $\{A,B\}$ with $A\sm B$ and $B\sm A$ both nonempty.

The separations of~$V\!$ even form a universe: if $\vr = (A,B)$ and $\vs = (C,D)$, say, then $\vr\vee\vs := (A\cup C, B\cap D)$ and $\vr\wedge\vs := (A\cap C, B\cup D)$ are again separations of~$V\!$, and are the supremum and infimum of $\vr$ and~$\vs$, respectively.

\subsection{Orienting a separation system}\label{subsec:orient}

Given a separation system $(\vS,\le\,,\!{}^*)$, a subset $O\sub\vS$ is an \emph{orientation} of $S$ (and of~$\vS$) if $O \cup O^\ast = \vS$ and $\lvert O\cap \{\vs,\sv \} \rvert = 1$ 
for all $s \in S$. Thus, $O$~contains exactly one orientation of every separation in~$S$. For subsets $R\sub S$ we say that $O$ {\em induces\/} and {\em extends\/} the orientation $O\cap\vR$ of~$R$, and thereby {\em orients\/}~$R$.

A set $O\sub\vS$ is {\em consistent\/} if there are no distinct $r,s \in S$ with orientations $\vr < \vs$ such that $\rv, \vs \in O$. Consistent orientations of~$\vS$ contain all its trivial separations.%
   \COMMENT{}

We say that $s\in S$ \emph{distinguishes} two orientations $O,O^\prime$ of subsets of~$S$ if $s$ has orientations $\vs \in O$ and $\sv \in O'$. (We then also say that $\vs$ and~$\sv$ themselves distinguish $O$ from~$O'$.) The sets $O,O'$ are {\em distinguishable\/} if neither is a subset of the other: if there exists some $s\in S$ that distinguishes them. A~set $T\sub S$ {\em distinguishes}~$O$ from~$O'$ if some $s\in T$ distinguishes them, and $T$ {\em distinguishes\/} a set~$\O$ of orientations of subsets of~$S$ if it distinguishes its elements pairwise.

Distinct orientations of the same separation system~$S$ are distinguished by every $s\in S$ on which they disagree, i.e.\ for which one of them contains~$\vs$, the other~$\sv$. Distinct orientations of different separation systems, however can be indistinguishable. For $R\subsetneq S$, for example, the orientations of~$S$ will be indistinguishable from the orientations of~$R$ they induce.

\subsection{Tree sets of separations}\label{subsec:nested}

Given a separation system $(\vS,\le\,,\!{}^*)$, two separations $r,s\in S$ are {\em nested\/} if they have comparable orientations; otherwise they \emph{cross}. Two oriented separations $\vr,\vs$ are {\em nested\/} if $r$ and~$s$ are nested. We say that $\vr$ {\em points towards\/}~$s$, and $\rv$ {\em points away from\/}~$s$, if $\vr\le\vs$ or $\vr\le\sv$. Then two nested oriented separations are either comparable, or point towards each other, or point away from each other. 

A~set of separations is {\em nested\/} if every two of its elements are nested. Two sets of separations are {\em nested\/} if every element of the first set is nested with every element of the second. A {\em tree set\/} is a nested separation system without trivial or degenerate elements.%
   \COMMENT{}
   A~tree set is {\it regular\/} if none of its elements is small~\cite{TreeSets}. When $\vT\sub\vS$ is a tree set, we also call $T\sub S$ a {\em tree set\/} (and {\em regular\/} if~$\vT$ is regular).

For example, the set of orientations $(u,v)$ of the edges $uv$ of a tree~$T$ form a regular tree set with respect to the involution $(u,v)\mapsto (v,u)$ and the \emph{natural partial ordering} on~$\vec E(T)$: the ordering in which $(x,y) < (u,v)$ if $\{x,y\}\ne\{u,v\}$%
   \COMMENT{}
   and the unique $\{x,y\}$--$\{u,v\}$ path in $T$ joins $y$ to~$u$.
   The oriented bipartitions of~$V(T)$ defined by deleting an edge of~$T$ form a tree set isomorphic to this one.%
   \COMMENT{}

Note that, in this latter example, the nodes of the tree~$T$ correspond bijectively to the consistent orientations of its edge set: orienting every edge towards some fixed node~$t$ is consistent, and conversely, every consistent orientation of the edges of a tree has a unique sink~$t$ towards which every edge is oriented.
Given an arbitrary tree set of separations, we may thus think of its consistent orientations as its `nodes', and of its elements as edges between these nodes, joining them up into a graph-theoretical tree. 

Every consistent orientation of a separation system~$\vS$ can be recovered from the set~$\sigma$ of its maximal elements: it is precisely the down-closure of~$\sigma$ in~$\vS$. Hence the consistent orientations of a tree set can be represented uniquely by just their sets of maximal elements. Such a set~$\sigma$ will always be a star of separations. The stars $\sigma\sub\vT$ in a tree set~$\vT$ of this form, those for which $\vT$ has a consistent orientation whose set of maximal elements is precisely~$\sigma$, are the {\em splitting stars\/} of the tree set~$\vT$ or of~$T$. The splitting stars in the set of oriented edges of a graph-theoretical tree are precisely the stars $\sigma_t$ of all the edges at a fixed node~$t$, oriented towards~$t$.~\cite{TreeSets}

\subsection{Profiles of separation systems}\label{subsec:profiles}

Let $(\vS,\le\,,\!{}^*)$ be a separation system inside a universe $(\vU,\le\,,\!{}^*,\vee,\wedge)$.%
   \COMMENT{}
 An orientation $P$ of~$S$ is a \emph{profile} (of~$S$ or~$\vS$) if it is consistent  and satisfies
 $$\text{For all $\vr,\vs \in P$ the separation $\rv \land \sv = (\vr\lor\vs)^*$ is not in $P$.}\eqno\rm(P)$$
Thus if $P$ contains $\vr$ and~$\vs$ it also contains $\vr\lor\vs$, unless $\vr\lor\vs\notin\vS$.%
  \COMMENT{}

Most natural orientations of a separation system that orient all its elements `towards' some large coherent structure (see the Introduction) will be both consistent and satisfy~(P): two oriented separations pointing away from each other (as in the definition of consistency) will hardly both point towards that structure, and similarly if $\vr$ and $\vs$ both point to the structure then their supremum~$\vr\lor\vs$ will do so too if it is in~$\vS$: otherwise the structure would be squeezed between $\vr$ and $\vs$ on the one hand, and $(\vr\lor\vs)^* = \rv\land\sv$ on the other, and could hardly be large. In our intended application to images, all orientations of some set of dividing lines of the picture towards one of its features will trivially satisfy~(P) and thus be a profile.

A profile is {\em regular\/} if it contains no separation whose inverse is small. Since all our nontrivial separations will be proper, all our profiles will be regular.%
   \COMMENT{}

Note that every subset $Q$ of a profile of~$\vS$ is a profile of~$\vR = Q\cup Q^*\sub\vS$.%
   \COMMENT{}
   Put another way, if $P$ is a profile of~$S$ and $R\sub S$, then $P\cap\vR$ is a profile of~$R$, which we say is {\em induced\/} by~$P$.

When $\vU$ is a universe of separations, we call the profiles of the separation systems $\vS_k\sub\vU$ of its separations of order~$<k$ the {\em $k$-profiles in\/}~$\vU$. For $\ell < k$, then, every $k$-profile $P$ in~$\vU$ induces an~$\ell$-profile.

\subsection{The tangle-tree theorem for profiles}\label{subsec:TTThm}

We are nearly ready to state the tangle-tree theorems for profiles. These come as a pair: in a canonical version, and in a refined, but non-canonical, version.

The canonical version finds, in any submodular universe~$\vU$ of separations, a `canonical' tree set of separations that distinguishes all the `robust and regular' profiles in~$\vU$. Here, a tree set $\vT\sub\vU$ is {\em canonical\/} if  the map $\vU\mapsto \vT$ commutes with isomorphisms of universes of separations, which are defined in the obvious way~\cite{ProfilesNew}. The definition of `robust' is more technical, but we do not need it: all we need is the following immediate consequence of its definition:

\begin{LEM} {\rm\cite{ProfilesNew}}
Every profile of a set of bipartitions of a set is robust.\qed
\end{LEM}%
   \COMMENT{}

Finally, the tree set $T$ found by the tangle-tree theorems will distinguish the relevant profiles $P,P'$ {\em efficiently\/}: from among all the separation in~$S$ that distinghish $P$ from~$P'$, the set $T$ will contain one of minimum order.

Here, then, is the canonical tangle-tree theorem for profiles in its simplest form~\cite[Corollary~3.8]{ProfilesNew}:

\begin{THM}\label{CTTThm} {\rm (The canonical tangle-tree theorem for profiles)}\\
Let $\vU = (\vU,\le\,,\!{}^*,\vee,\wedge,|\ |)$ be a submodular universe of separations. There is a canonical%
   \COMMENT{}
   regular tree set $T\sub \vU$ that efficiently distinguishes all the distinguishable robust and regular profiles in~$\vU$.
\end{THM}

Note that Theorem~\ref{CTTThm} finds a tree set of separations that distinguishes {\em all\/} the profiles in~$\vU$ that we can ever hope to distinguish. Sometimes, however, we only need to distinguish some of them, say those in some particular set~$\P$ of profiles. In this case a smaller tree set $T$ of separations will suffice. The non-canonical tangle-tree theorem says that, if we do not care about canonicity, we can make $T$ as small as we theoretically can: so small that deleting any one separation from~$T$ will lose it its property of distinguishing~$\P$.%
   \COMMENT{}

As a nice consequence of this minimality, the non-canonical tangle-tree theorem assigns the profiles it distinguishes, those in~$\P$, to the `nodes of the tree' which $T$ defines. Since $T$ is just a tree set these are, of course, undefined. However, we saw in Section~\ref{subsec:nested} that {\em if\/} a tree set of separations is the edge set of a tree, or equivalently the set of nested bipartition of the vertices of a tree induced by its edges, then the nodes of this tree correspond to exactly the consistent orientations of its edges.

As $T\sub U$, every profile in~$\vU$ will orient some of the separations in~$T$, and it will do so consistently. Such a consistent partial orientation of~$T$ will, in general, extend to many consistent orientations of all of~$T$. The partial orientations induced by the profiles in~$\P$, however, extend uniquely, and thus lead to a bijection mapping $\P$ to the `nodes' of~$T$, its consistent orientations:

\begin{THM}\label{TTThm} {\rm (The non-canonical tangle-tree theorem for profiles~\cite{ProfilesNew})}\\
   Let $\vU = (\vU,\le\,,\!{}^*,\vee,\wedge,|\ |)$ be a submodular universe of separations. For every set $\P$ of pairwise distinguishable robust and regular profiles in~$\vU$ there is a regular%
   \COMMENT{}
   tree set $T\subseteq \vU$ of separations such that:
   \begin{enumerate}[{\rm (i)}]\itemsep0pt
      \item every two profiles in~${\P}\!$ are efficiently distinguished by some separation~in~$T$;
      \item no proper subset of~$T$ distinguishes~$\P$;
      \item for every $P\in\P$ the set $P\cap\vT$ extends to a unique consistent orientation\penalty-200\ $O_P$ of~$T$. This map $P\mapsto O_P$ from $\P$ to the set $\O$ of consistent orientations of~$T$ is bijective.
   \end{enumerate}
\end{THM}

Theorem~\ref{TTThm} starts with a seemingly stronger premise than Theorem~\ref{CTTThm} does: we must fix a set $\P$ of pairwise distinguishable profiles that we wish to distinguish. The set of {\em all\/} robust and regular profiles, which Theorem~\ref{CTTThm} speaks about, clearly does not satisfy this premise. However, the tree set~$T$ which Theorem~\ref{TTThm} returns will still distinguish every two such profiles that are distinguishable at all (as in Theorem~\ref{CTTThm}): we just have to apply Theorem~\ref{TTThm} with $\P$ the set of all the {\em maximal\/} robust and regular profiles. Then for any distinguishable  robust and regular profiles $P,Q$ there are profiles $P'\supe P$ and $Q'\supe Q$ in~$\P$, and any separation that distinguishes $P'$ from~$Q'$ efficiently will also distinguish $P$ from~$Q$. And the tree set $T$ returned by Theorem~\ref{TTThm} will contain such a separation.

\subsection{Tangles of separation systems, and duality}\label{subsec:tangles}

A set $\sigma$ of oriented separations in some given separation system $(\vS,\le\,,\!{}^*)$ is a \emph{star of separations} if they point towards each other: if $\vr\le\sv$ for all distinct ${\vr,\vs\in\sigma}$. In particular, stars of separations are nested. They are also consistent: if $\rv,\vs$ lie in the same star we cannot have $\vr < \vs$, since also $\vs\le \vr$ by the star property.

An \emph{${S}$-tree} is a pair $(T,\alpha)$ of a tree $T$\ and a function $\alpha\colon\vec E(T)\to \vS$ such that, for every edge $xy$ of~$T$, if $\alpha(x,y)=\vs$ then $\alpha(y,x)=\sv$. It is an $S$-tree {\em over $\F\sub 2^{\vec S}$} if for every node $t$ of~$T$ we have $\alpha(\vec F_t)\in\F$, where
 $$\vec F_t := \{(x,t) : xt\in E(T)\}.$$

\noindent
 We shall call the set $\vec F_t\sub\vec E(T)$ the {\em oriented star at~$t$} in~$T$. Usually, $\F$~will be a set of stars in~$\vS$.

An $\F$-{\em tangle\/} of~$S$ is a consistent orientation of~$S$ that {\em avoids\/}~$\F$, that is, has no subset $\sigma\in\F$. Since all consistent orientations of~$\vS$ contain all its trivial elements, any $\F$-tangle of~$S$ will also be an $\F'$-tangle if $\F'$ is obtained from~$\F$ by adding all {\em co-trivial singletons\/}, the sets $\{\rv\}$ with $\vr$ trivial in~$S$.%
   \COMMENT{}
   Sets~$\F$ that already contain all co-trivial singletons will be called {\em standard\/}.

The tangle duality theorem for abstract separation systems~$\vS$ says that, under certain conditions, if $\F\sub 2^{\vec S}$ is a standard set of stars there is either an $\F$-tangle of~$S$ or an $S$-tree over~$\F$, but never both. The conditions, one on~$\vS$ and another on~$\F$, are as follows.

A separation $\vso\in\vS$ \emph{emulates} another separation $\vr\le\vso$ if every%
   \footnote{Technically, we also need this only for $\vs\ne\rv$, and only when $\vr\ne\rv$. Since all the separations we consider in this paper are proper, however, both these will hold automatically: if $\vs=\rv$ then $\vr\le\vs = \rv$ is small, and likewise if $\vr=\rv$.}
   $\vs\ge\vr$ satisfies $\vs\lor\vso\in\vS$. The separation system $\vS$ is \emph{separable} if for every two nontrivial $\vr,\rvdash\in\vS$ such that $\vr\le\vrdash$ there exists an $s_0\in S$ with an orientation $\vso$ that emulates~$\vr$ and its inverse $\svo$ emulating~$\rvdash$. This may sound complicated, but we do not need to understand the complexity of this here: the separation system in our intended application will easily be seen to be separable, which will allow us to apply the tangle-tree duality theorem and benefit from what it says, which does not involve the notion of separability.

Similarly, $\F$~has to satisfy a complicated condition, which we shall have to check but do not need to understand. Let $\vr\le\vso\le\vrdash$ be as above, with $\vso$ emulating~$\vr$ and $\svo$ emulating~$\rvdash$. For every $\vs\in\vS$ with $\vr\le\vs$ ($\ne\rv$), let
 $$f\shift(\!\vr)(\vso) (\vs) := \vs\vee\vso\quad\text{and}\quad f\shift(\!\vr)(\vso) (\sv) := (\vs\vee\vso)^*.$$
 We say that $\F$ is {\em closed under shifting\/} if, for all such $\vr,\vrdash$ and~$\vso$, the image under $f\shift(\!\vr)(\vso)$ of every star $\sigma\in\F$ containing a separation $\vs\ge\vr$ again lies in~$\F$. (Note that $f\shift(\!\vr)(\vso) (\vsdash)$ is defined for all $\vsdash\in\sigma$, since $\vr\le\vs\le\svdash$ by the star property.)

The premise of the tangle duality theorem for abstract separation systems contains a condition that combines these two assumptions into one; it requires $\vS$ to be `$\F$-separable'. As before, we do not need to understand what this means, since it is an easy consequence of the two assumptions just discussed:

\begin{LEM}\label{Fsep} {\rm\cite{TangleTreeGraphsMatroids}}
If $\vS$ is separable and $\F$ is closed under shifting, then $\vS$ is $\F$-separable.
\end{LEM}

Here, then, is the tangle duality theorem for separation systems~\cite[Thm~4.4]{TangleTreeAbstract}:

\begin{THM} {\rm (Tangle duality theorem for abstract separation systems)}\label{TDThm}\\
  Let $(\vU,\le\,,\!{}^*,\vee,\wedge)$ be a universe of  separations containing a separation system $(\vS,\le\,,\!{}^*)$. Let $\F\sub 2^\vS$ be a standard set of stars. If $\vS$ is $\F$-separable,%
   \COMMENT{}
   exactly one of the following assertions holds:
\begin{enumerate}[\rm(i)]\itemsep0pt
  \item There exists an ${S}$-tree over $\F$.
  \item There exists an $\F$-tangle of~$S$.
  \end{enumerate}
\end{THM}

\section{Images as collections of tangles}\label{sec:apps}

We finally come to apply our theory to image recognition and compression. The basic idea is as follows. In our attempt to identify, or perhaps just to store, the main features of a given image, we consider lines cutting the canvas in two. Lines that cut across a highly coherent region, a region whose pixels are all similar, will be a bad line to cut along, so we assign such lines a high `order' when we view it as a separation of the picture.

The coherent regions of our image should thus correspond to tangles, or profiles,%
   \footnote{In informal discourse, we often say `tangle' instead of `profile', although formally tangles are just a particular type of profile. The sloppiness is justified, however, by the fact that every $k$-profile in a universe of bipartitions of a set, as considered here, is an $\F$-tangle of~$S_k$ for the set $\F$ of stars violating the profile condition~(P). See the proof of Theorem~\ref{TDapplied}.}
   of some order, the higher the greater their perceived coherence. Or more precisely: we {\em define\/} `regions' as tangles in the universe of separations given by those cuts.%
   \footnote{Actually, regions will be equivalence classes of tangles; see later.}
   The novelty of this approach is that, as with tangles in a graph, we make no attempt to assign individual pixels to one region or another.

However, we shall be able to outline those regions. If we wish to outline some particular set $R$ of regions represented by a set $\P$ of pairwise distinguishable profiles,%
   \COMMENT{}
   we apply the non-canonical tangle-tree theorem to ~$\P$. The theorem will return a nested set~$T$ of `separations', which corresponds to a set $L$ of cutting lines. As $T$ is nested, these lines will not cross each other. There are enough of them that for every pair of regions in~$R$ we shall find one that runs between them, separating them in the picture. Since $T$ will distinguish~$\P$ efficiently, we shall even find such a line in~$L$ whose order is as small as that of any line separating these two regions in the entire picture: it will cut through as few `like' pixels as possible for any line separating these two regions. Finally, since the profiles in~$\P$ correspond bijectively to the consistent orientations,%
   \COMMENT{}
   and hence to the splitting stars, of~$T$, the regions in $R$ will correspond to what these splitting stars translate to: the set of all lines in $L$ that are closest to that region, thus outlining it.%
   \COMMENT{}

Finally, we have the tangle duality theorem. It tells us that if a separation system, such as the set $S_k$ of all cutting lines through our picture that we gave an order~$<k$, has no $\F$-tangle for some set $\F$ of stars in~$\vec S_k$, it has an $S_k$-tree over~$\F$. This means that our picture either has a region such that lines of order~$<k$ cannot cut across it, or else we can cut up the entire picture into single pixels by non-crossing lines all of order~$<k$.

Note that this lends some definiteness to our intended definion of `regions' as tangles: although it is fuzzy in that single pixels cannot be assigned to regions, it is also hard in the sense that the picture will, or will not, contain a `region of coherence~$k$': there is nothing vague about the existence or non-existence of an $\F$-tangle of~$S_k$. Up to a certain `resolution' (value of~$k$)%
   \COMMENT{}
   we shall see more and more tangles, each refining earlier ones; but at some point there will only be the `focused' tangles, those pointing to some fixed pixel. This, then, means that the quality of the image does not support any higher `resolution' than~$k$. And the tangle duality theorem will tell us at which resolution this happens. As a spin-off, we obtain a mathematically rigorous definition of the adequacy of a given resolution for a given pixellated image.

So here is our formal setup. Consider a {\em flat\/} 2-dimensional cell complex $X$ {\em of squares\/}: a CW-complex whose 2-cells are each bounded by a set of exactly four 1-cells, and in which every 1-cell lies on the boundary of at most two 2-cells. Think of $X$ as a quadrangulation of a closed surface with or without boundary.

We shall call the 2-cells of~$X$ {\em pixels\/} and its 1-cells {\em edges\/}. We write $\Pi$ for the set of pixels of~$X$, and $E$ for the set of its edges. An edge that lies on the boundary of two pixels is said to {\em join\/} these pixels.

A {\em picture\/} on~$X$ is a map $\pi\colon\Pi\to 2^n$, where we think of the elements of $2^n = \{0,1\}^n$ as giving each of $n$ parameters a value 0 or~1. Think of these parameters as brightness, colour and so on. We use addition on~$2^n$ modulo~2 and denote your favourite norm on~$2^n$ by~$\|~\|$. To an edge $e$ joining two pixels $p$ and~$q$ we assign the number
 $$\delta(e) := \|\pi(q) - \pi(p)\|;$$
 note that this is well defined, regardless of which of the two pixels is $p$ and which is~$q$. We call $\delta(e)$ the {\em weight\/} of the edge~$e$, and $\delta$ a {\em weighting\/} of the picture~$\pi$.

The {\em boundary\/} of a set $A\sub\Pi$ of pixels is the set
 $$\partial A := \{\,e\in E\mid e\text{ joins a pixel }p\in A\text{ to a pixel }q\notin A\}.$$%
   \COMMENT{}

We are now ready to set up our universe~$\vU$ of separations of a given picture~$\pi$. We let $\vU := 2^\Pi$.%
   \COMMENT{}
    Thus, every set $A$ of pixels is an oriented `separation'. The partial ordering~$\le$ on~$\vU$ will be~$\supe$, the involution~* on~$\vU$ will be complementation in~$\Pi$. The join~$\lor$ in~$\vU$ is~$\cap$, and the meet~$\land$ is~$\cup$. If desired, we can think of a separation~$A$ as the oriented bipartition $(A^*,A)$ of~$\Pi$.

Our universe $\vU$ contains no nontrivial small separations, since $A^*\not\sub A$ for all $A\subsetneq\Pi$. Every tree set in~$\vU$ and every profile in $\vU$%
   \COMMENT{}
   will therefore be regular. Also, all profiles and sets of pairwise distinguishable profiles in~$\vU$ will be robust in the sense of~\cite{ProfilesNew}; see the definitions there.%
   \COMMENT{}

We define  the {\em order\/} of a separation $A\in\vU$ as
 $$|A| := \sum\big\{\, N - \delta(e)\mid e\in\partial A\,\big\},$$
 where $N$ is some fixed integer just large enough to make all $|A|$ positive. Formally, we fix not just $X$ but the pair $(X,N)$ at the start, calling it a {\em canvas\/}.

Notice that $|A| = |A^*|$, as required for an order function. Moreover,

\begin{LEM}
The order function $A\mapsto |A|$ is submodular.
\end{LEM}

\begin{proof}
We have to show that $|A\cap B| + |A\cup B|\le |A| + |B|$ for all $A,B\sub\Pi$. We shall prove that every $e$ such that $\delta(e)$ is counted on the left will also have its $\delta(e)$ counted on the right, and if it is counted in both sums on the left it is also counted in both sums on the right.

If $e\in \partial(A\cap B)$, then $e$ joins a pixel $p\in A\cap B$ to a pixel~$q$ that fails to lie in~$A$ or which fails to lie in~$B$. In the first case, $e\in\partial A$, in the latter case we have $e\in\partial B$. Since $\partial (A\cup B) = \partial (A^*\cap B^*)$, the same holds for edges in $\partial (A\cup B)$: every such edge lies in $\partial(A^*) = \partial A$ or in $\partial(B^*) = \partial B$.

Finally, if $e$ is counted twice on the left, i.e., if $e\in\partial (A\cap B)$ as well as $e\in\partial (A\cup B) = \partial (A^*\cap B^*)$, then $e$ joins some pixesl $p\in A\cap B$ to some other pixel, and it also joins some $q\in A^*\cap B^*$ to some other pixel. As $A\cap B$ and $A^*\cap B^*$ are disjoint, we have $p\ne q$, so $p$ and $q$ are precisely the two pixels on whose boundary $e$ lies. But this means that $e\in\partial A$%
   \COMMENT{}
   as well as $e\in\partial B$,%
   \COMMENT{}
   so $e$ is counted twice also on the right.
\end{proof}

It remains to translate the terminology of tangles into some more suitable language for our application. We shall do this separately for our two theorems: for the tangle-tree theorem we use profile as `tangles', while for the tangle duality theorem we shall use $\F$-tangles for some natural~$\F$. It is possible in more than one way to unify these into a single notion of `tangle' to which both theorems apply~\cite{AbstractSepSys, AbstractTangles, ProfileDuality}, but to keep things simple let us here work with the slightly different notions used in our two theorems.

As before, we write $\vec S_k := \{\vs\in\vU: |s|<k\}$, and the {\em $k$-profiles\/} in~$\vU$ are the profiles of~$S_k$. 

A $k$-profile~$P$ is {\em focused\/} if there exists a pixel~$p$ such that $\{p\}\in P$. It is {\em principal\/} if there exists a pixel $p$ such that $P = \{\,A\subset\Pi\mid p\in A\}\cap\vec S_k$. Focused profiles are always principal, but not conversely.%
   \COMMENT{}
   Note that every pixel~$p$ defines a focused $k$-profile for {\em every\/} integer $k>|\{p\}|$.

We consider two profiles $P,Q$ in~$\vU$ as {\em equivalent\/} if, possibly after swapping their names, $P$~is a $k$-profile and $Q$ is an $\ell$-profile for some $\ell\le k$ and every $k'$-profile with $\ell\le k'\le k$%
   \COMMENT{}
   that induces~$Q$ is in turn induced by~$P$. This is indeed an equivalence relation.%
   \COMMENT{}

A {\em region\/} of~$\pi$ is an equivalence class of profiles in~$\vU$ that contains no focused profile. The reason for disregarding classes represented by a focused profile is not only that these will be uninteresting and easy to deal with if necessary; it is also that our notion of the `cohesion' and `visiblity' of a region (see below) would otherwise be undefined.%
   \footnote{Intuitively, a region will be more visible if it contains $k$-profiles for bigger~$k$. Regions containing focused profiles would have infinite visibility, which would be inadequate and technically cumbersome.}

Equivalent profiles are indistinguishable, because one induces the other. The converse is not normally true: if $P$ as above induces~$Q$ but $Q$ is also induced by another $k$-profile~$P'$, say, then neither $P$ nor~$P'$ will be equivalent to~$Q$. We think of their regions as different refinements of the region represented by~$Q$.

Indeed, let us call two regions {\em distinguishable\/} if they contain distinguishable profiles,%
   \COMMENT{}
   and {\em indistinguishable\/} otherwise. Note that if a profile from a region~$\sigma$ induces a profile in a region $\rho\ne\sigma$ then every profile in~$\sigma$ induces every profile in~$\rho$.%
   \COMMENT{}
   A~region $\sigma$ {\em refines\/} a region~$\rho$ if the profiles in~$\sigma$ induce those in~$\rho$.%
   \COMMENT{}
   Two such regions are distinct but indistinguishable.

The {\em complexity\/} and the {\em cohesion\/} of a region~$\rho$, respectively, are the smallest and the largest $k$ such that $\rho$ contains a $k$-profile.%
   \COMMENT{}
   Its {\em visibility\/} is the positive difference between its cohesion and its complexity. It is an easy (and pretty) exercise to show that the maximum in the definition of cohesion exists.%
   \COMMENT{}

For example, consider a region~$\rho$ of complexity~$c$ and cohesion~$d$. Pick $k < c$, and consider the unique $k$-profile $P$ induced by the elements of~$\rho$. This profile represents a region $\sigma$ which $\rho$ refines. Since $k<c$ and therefore $P\notin\rho$, we have $\sigma\ne\rho$. Hence $\sigma$ is also refined by some region~$\rho'$ that does not, in turn, refine~$\rho$. Similarly, there is either no $k$-profile for any $k>d$ that induces the profiles%
   \COMMENT{}
   in~$\rho$, or there are several, at least two of these disagree on~$S_d$,%
   \COMMENT{}
   and any such two represent distinct regions refining~$\rho$.%
   \COMMENT{}

A {\em line\/} is an unoriented nontrivial separation. Thus, formally, a line $\ell$ is an unordered pair $\{A,A^*\}$ with $A$ a set of pixels that is neither empty nor all of~$\Pi$. Since $\partial A = \partial A^*$, the line $\ell$ determines the set $\partial A = \partial A^*$ of edges `between' $A$ and~$A^*$, and it can be helpful to imagine lines as boundaries. However, boundaries can have more than one connected component, and we make no claim as to whether or not the pair $\{A,A^*\}$ can be reconstructed from this common boundary.%
   \footnote{It can, but we do not need this~-- it just complicates matters to reduce lines to boundaries.}%
   \COMMENT{}

The {\em order\/} of a line $\ell = \{A,A^*\}$ is the number $\ell:= |A| = |A^*|$. A~nested set of lines will be called {\em laminar\/}. It is {\em canonical\/} if it is invariant under the automorphisms of the universe~$\vU = (\vU,\le\,,\!{}^*,\vee,\wedge,|\ |)$, which clearly act naturally on its lines.

A line~$\ell$ {\em separates\/} two regions $\rho,\rho'$ if it distinguishes the profiles in~$\rho$ from those in~$\rho'$.%
   \COMMENT{}
   It separates them {\em efficiently\/} if no line in~$U$ of lower order than~$|\ell|$ separates $\rho$ from~$\rho'$. Note that a line that distinguishes two profiles will separate the regions which these profiles represent.%
   \COMMENT{}
   In particular, a set of lines that distinguishes all the distinguishable profiles in~$\vU$ will separate all the distinguishable regions of~$\pi$.%
   \COMMENT{}

The {\em outline\/} of a region~$\rho$ in a set $L$ of lines is the set of maximal elements of~$P\cap\vL$, where $P$ is the unique $k$-profile in~$\rho$ for $k$ the complexity of~$\rho$.%
   \COMMENT{}
   If $L$ is a tree set, the outline of $\rho$ in~$L$ will be a splitting star of~$L$.%
   \COMMENT{}

Here, then, are our tangle-tree theorems translated to the world of pictures. First, the canonical version:%
   \COMMENT{}

\begin{THM}\label{Canonicalapplied}
For every picture $\pi$ on a canvas~$(X,N)$ there is a canonical lami\-nar set $L$ of lines that efficiently separates all distinguishable regions of~$\pi$.\qed%
   \COMMENT{}
\end{THM}

Next, the refined but non-canonical version:

\begin{THM}\label{Noncanonicalapplied}
For every set $R$ of pairwise distinguishable regions of a picture~$\pi$ on a canvas~$(X,N)$ there is a laminar set $L$ of lines that efficiently separates all the regions in~$R$, and which is miminal in the sense that no proper subset of~$L$ separates all the regions in~$R$ (efficiently or not).

The splitting stars of~$L$ are precisely the outlines in~$L$ of the regions in~$R$.\qed
\end{THM}

For the tangle duality theorem we need one more definition. A {3\em-star\/} in~$\vU$ is a star with exactly 3 elements. A star $\sigma\sub\vU$ is {\em void\/} if $\bigcap\sigma = \es$. A~star $\sigma = \{\{p\}\}$ for some $pÊ\in\Pi$ is a {\em single pixel\/}.%
   \COMMENT{}

\begin{THM}\label{TDapplied}
For every picture $\pi$ on a canvas~$(X,N)$ and every integer~$k>0$, either $\pi$ has a region of cohesion at least~$k$,%
   \COMMENT{}
   or there exists a laminar set~$L$ of lines of order~$<k$ all whose splitting stars are void 3-stars%
   \COMMENT{}
    or single pixels. For no picture do both these happen at once.
\end{THM}

\begin{proof}
Let $\F$ be the set of all void stars with up to 3 elements and all single pixels. Note that $\F$ is standard, since co-trivial singletons are void stars. Let us show that $S_k$ has an $\F$-tangle if and only if $\pi$ has a region of cohesion at least~$k$.

Assume first that $S_k$ has an $\F$-tangle, $P$~say. We start by showing that $P$ is a profile. To verify the profile condition~(P), let $\vr,\vs\in P$ be given; we have to show that $\rv\land\sv\notin P$. By submodularity, one of $\vr\land\sv$ and $\vs\land\rv$ lies in~$\vS_k$. Assume that $\vr\land\sv$ does; the other case is analogous. As $\vr\land\sv \le\vr\in P$, the consistency of $P$ implies that $\vr\land\sv \in P$. But $\{\vr\land\sv,\vs,\rv\land\sv\}$ is a void star, and hence in~$\F$. Since $P$ avoids~$\F$, we therefore have $\rv\land\sv\notin P$ as required.

Next, observe that the profile~$P$ is not focused: for every $p\in\Pi$ we have $\{\{p\}\}\in\F$ and hence $\{p\}\notin P$, by the definition of~$P$ as an $\F$-tangle.

The equivalence class $\rho$ of profiles in~$\vU$ which $P$ represents is therefore a region of~$\pi$. Since $P\in\rho$ is an orientation of~$S_k$, and hence a $k$-profile, $\rho$~has cohesion at least~$k$.

Conversely, assume that $\pi$ has a region~$\rho$ of cohesion at least~$k$. Then $\rho$ contains a $k'$-profile~$P$ for some $k'\ge k$. Since any void star in~$P$ of up to three separations constitutes a violation of~(P),%
   \COMMENT{}
   which $P$ satisfies,%
   \COMMENT{}
   there are no such stars in~$P$. If $P$ contained a single pixel~$\{\{p\}\}$ as a subset, it would contain the separation $\{p\}$ as an element and hence be focused. But since $\rho\owns P$ is a region, $P$~is not focused. Hence $P$ is an $\F$-tangle of~$S_k$.

We have shown that one of the two conditions in the tangle duality theorem, the existence of an $\F$-tangle of~$S_k$, translates to the existence of a region of cohesion at least~$k$ in~$\pi$, as desired. The other condition is that there exists an $S_k$-tree over~$\F$. By contracting and pruning this $S_k$-tree in an obvious way~\cite{TreeSets}, we can turn it into an $S_k$-tree $(T,\alpha)$ over the set $\F'\sub\F$ of just the void 3-stars and single pixels in~$\F$, i.e., which uses neither void 2-stars nor void 1-stars.%
   \COMMENT{}
   The unoriented separations in the image of~$\alpha$ then form the desired%
   \COMMENT{}
   laminar set $L$ of lines of order~$<k$. Conversely, it is shown in~\cite{TreeSets} that from any tree set $S$ of separations without splitting stars in~$\F'$ one can construct an $S$-tree $(T,\alpha)$ over~$\F'$, and thus also over~$\F$, with $S$ the image of~$\alpha$. Hence the second condition in the tangle duality theorem also translates as desired.

In order to be allowed to apply the tangle duality theorem, with $\vS := \vU$, we still need to check that $\vS_k$ is separable and that $\F$ is closed under shifting; cf.\ Lemma~\ref{Fsep}.%
   \COMMENT{} The fact that $\vS_k$ is separable is a direct consequence of submodularity and proved in~\cite{TangleTreeGraphsMatroids}. The fact that $\F$ is closed under shifting is an easy consequence of the fact that $\vU$ consists of bipartitions of a set,~$\Pi$. Indeed, consider a star $\sigma = \{A_0,A_1,A_2\}\in\F$ as in the definition of shifting.%
   \footnote{If $\sigma$ contains only one or two separations, ignore $A_1$ or~$A_2$ as appropriate.}
   If $\vso=C$ and $\vr\le A_0$, say (see there), then $A_1$ and $A_2$ shift to supersets $A_1\cup C^*$ and $A_2\cup C^*$, respectively, but $A_0$ shifts to $A_0\cap C$. Hence if $A_0,A_1,A_2$ have an empty overall intersection then so do their shifts: void stars shift to void stars. Similarly, a single pixel~$\{\{p\}\}$ shifts to~$\{\{p\}\cap C\}$, which is either equal to $\{p\}$ or a co-trivial singleton~$\{\es\}$, and hence again in~$\F$.
\end{proof}

We remark that the choice of $\F$, which led to the theorem above, is but a minimal one that makes the tangle duality theorem applicable. We could choose to add more stars of separations to~$\F$: any stars $\{A_1,\dots,A_n\}$ whose `interiors' $A_1^*\cap\ldots\cap A_n^*$ are deemed to be too small to be home to an interesting tangle.

\section{Potential applications}

As mentioned in the introduction, tangles are intrinsically large objects from a complexity point of view: a tangle of order~$k$, or a $k$-profile, has to orient every separation of order less than~$k$, of which there can be many and which also have to be found. But while tangles are complex things, the laminar sets of lines given by Theorems~\ref{Canonicalapplied} and~\ref{Noncanonicalapplied} are not. So even if generating them turned out to be computationally hard (which is not clear; see below), once they are found they offer a much compressed version of the essence of the picture.

Just how well a few lines separating the main regions of a picture can convey its essence can be gleaned, for example, from caricatures: it seems that such lines, more than anything else, trigger our visual understanding and recognition.

Moreover, there is no reason to believe that these laminar sets of lines cannot be generated efficiently from a given picture, taking advantage of the additional structure that pictures offer over abstract separation systems.

For a start, we could economise by generating the profiles together with the lines that distinguish them,%
   \COMMENT{}
   perhaps in increasing order of those lines.%
   \footnote{There will be a trade-off between the synergy of generating lines of different orders together and the benefits of possibly not even having to look at lines of high order; see Section~\ref{subsec:recognition}.}
   Once we have generated all the lines up to some desired order, their profiles can be obtained relatively easily from the poset they form. Indeed, recall that a profile is a consistent orientation of the set $S_k$ of lines (separations) of order less than~$k$ that satisfies condition~(P). With some book-keeping it should be possible to generate the poset $\vec S_k$ together with the lines in~$S_k$. Its profiles are then easily computed: we iteratively orient a line $s$ not yet oriented, as $\vs$ say, adding at the same time its down-closure in~$\vec S_k$ and all the oriented lines $\vr\lor\vs\in\vec S_k$ it generates with any $\vr$ selected earlier. This profile can then be stored, if desired, by remembering just its maximal elements, since it is precisely their down-closure in~$\vec S_k$.

\subsection{Image recognition}\label{subsec:recognition}

The canonical tangle-tree theorem derives from a given image a tree set of lines each labelled by its order. These tree sets are tantamount to trees with labelled edges~\cite{TreeSets}. Since they are generated canonically, trees coming from images of the same object will have large subtrees that are isomorphic, or nearly so, as edge-labelled trees. Hence, an image recognition algorithm might test for the presence of such subtrees, and return a `different' verdict for images of objects where such subtrees are not found. Those rare pairs of objects where such trees are found could be analysed in more detail and at greater cost.

The proof of the canonical tangle-tree theorem generates its tree set of separations in increasing order: low-order separations are included in the final tree set at an early stage, before higher-order separations are even considered. When we compare two pictures, and these differ already by their low-order regions, we will know this early in the algorithm and can terminate it at that stage.

\subsection{Compression}

The non-canonical tangle-tree theorem returns an even smaller tree set of lines that still separate all the distinguishable regions of an image. Again, these lines can be labelled by their order and thus give rise to an edge-labelled tree, which will be cheap to store.

In addition, we might reduce the number of lines further by limiting the regions in~$R$ we aim to distinguish (see Theorem~\ref{Noncanonicalapplied}) to those of high visibility. These regions will correspond to the most distinctive features of a picture, and it may well be enough to demarcate just these.

To re-create an approximation of the original picture from these few laminar lines, one can draw them on a canvas and fill the areas between them~-- those corresponding to the  splitting stars of the tree set, the nodes of the tree~-- with pixels similar to those in a small sample taken from the corresponding areas in the original picture when these lines were computed. For every such area, its pixels are likely to be similar, so even a small sample should suffice to smooth out the gradual differences that can still occur within such an area.%
   \footnote{The footprint under~$\pi$ of the pixels in a large such area may shift gradually. The fact that it comes from a region only means that, locally, we cannot cut through it by a line of low order: adjacent pixels tend to be similar, but they may change gradually, and a few atypical pixels can occur. Still, a small set of pixels sampled from the area should give a good representation.}

\subsection{Supported resolution: telling objects from noise}

The tangle duality theorem allows us to offer a mathematically rigorous definition of the maximum resolution%
   \COMMENT{}
   that a set $A\sub\Pi$ of pixels supports: the {\em largest~$k$ for which it admits an unfocused $k$-profile, i.e., has a region of cohesion at least~$k$.}%
   \COMMENT{}

For if we are interested in the potential features of an image handed to us as just a data set of pixels, then real features are likely to correspond to regions of cohesion at least some $k$ that we may specify, while areas not containing such a region will be unimportant background, or `noise', at this {\em resolution\/}~$k$.

Note that this notion of resolution does not refer to how small our pixels are: these are assumed to be constant. What it measures is up to what degree of similarity an area of the picture blending with, or inserted within, another area can still be identified as a feature of the picture with some objective degree of certainty. In a nutshell, `higher resolution' in our sense does not mean `more pixels' (which we consider as given) but `more features' read out of those pixels.

We shall discuss this in more detail in Section~\ref{sec:examples}, in the context of the example shown in Figures \ref{f4} and~\ref{f6}.

\subsection{Sending pictures through a noisy channel}

A central problem in information theory is how to send information, such as a picture, through a noisy channel in such a way that its essence can be retrieved at the other end. It would seem that tangles, being inherently fuzzy, can survive such distortions particularly well. Deleting or adding a few vertices or edges to or from a large grid will not change the fact that it induces a large-order tangle.

The most visible regions of a picture, and their relationship as to which refines which, should thus survive even the transmission of the picture by way of sending pixels: although their complexity and cohesion might change a little, highly visible regions will remain highly visible regions.%
   \COMMENT{}

\section{Examples}\label{sec:examples}

In this section we collect some examples that illustrate some of the ideas discussed in this
paper.

\begin{figure}[h]
  \center
    \includegraphics[scale=.3]{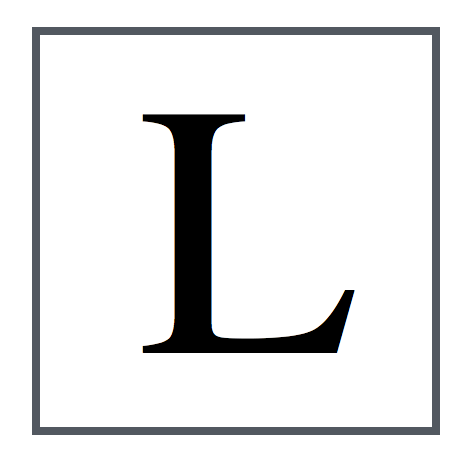}
\caption{A simple shape with one main region}
\label{f1}
\end{figure}

Consider the black and white picture of the letter~L shown in Figure~\ref{f1}. The function~$\pi$ assigns~1 to the black pixels on the canvas, and 0 to the white ones. Let~$\delta$ be the weighting that assigns~1 to an edge~$e$ if $e$ joins a black pixel to a white one, and~0 otherwise. The line $\ell=\{A,A^*\}$, where $A$ is the set of black pixels, separates the L from its background. If  $N=1$, then the line $\ell$ as well as the trivial separation $\Pi$ and its inverse~$\emptyset$ have order~0.

Since no line of order~0 has an edge in its boundary that joins two like pixels (both black or both white), because such an edge~$e$ would contribute~1 to the order of such a line, $\ell$~is the only line of order~0. It separates two regions of complexity~1: the region $\rho$ represented by the 1-profile $\{A\}$ for the letter~L, and the region represented by the 1-profile~$\{A^*\}$ for its background.%
   \COMMENT{}

Assuming that the L is 10 pixels wide where it is thinnest at the bottom, consider the 11-profile
 $$P = \{A, B_1,\dots,B_n, C_1,\dots, C_m\},$$
 where $B_1\supset\ldots\supset B_n$ are the subsets of~$A$ that contain the right serif%
  \COMMENT{}
   and have order~10, i.e., whose boundary contains only 10 (vertically stacked vertical) edges joining like pixels (which are both black), and the $C_i$ are the complements in~$\Pi$ of sets $D_i\sub\Pi$ of order at most~10 (when viewed as a separation) such that $|D_i\cap B_n|\le |D_i|$.%
   \COMMENT{}
  (The shape of L guarantees that this cannot hold for both $D_i$ and its complement.) These $D_i$ are precisely the sets of order at most~10 that contain fewer pixels from~$B_n$ than their complement does.%
   \COMMENT{}

It is easy to check that $P$ is indeed a profile.%
   \COMMENT{}
   It is unfocused, since none of its elements is a singleton~$\{p\}$. Indeed, if $p$ is white then $\{p\}\sub A^*$ does not lie in~$P$ by consistency, because $A$ does. If $p$ is black, then $\{p\}$ is one of those~$D_i$ defined above, so again $\{p\}\notin P$ by consistency, because $\{p\} = D_i = C_i^*$. Hence $P$ represents a region~$\sigma$.

This `serif region'~$\sigma$ thus has complexity at most~11.%
   \COMMENT{}
   We shall see in a moment that it has complexity exactly~11. Its outline in~$\vU$, therefore,%
   \COMMENT{}
   will be the set of maximal elements of~$P$: the separation $B_n$ and all the separations~$C = D^*$ with $D$ maximal among the~$D_i$.%
   \COMMENT{}

Now consider the 11-profile
 $$Q = \{A, B'_1,\dots,B'_n, C'_1,\dots, C'_{m'}\},$$
 where $B'_i:= A\sm B_i$, and the $C'_i$ are the complements in~$\Pi$ of sets $D_i\sub\Pi$ of order at most~10 such that $|D_i\cap B'_1|\le |D_i|$. Since $P$ and $Q$ are distinct profiles of the same order, $Q$~represents a region~$\tau$ distinct from~$\sigma$; we may think of it as corresponding to the vertical shaft of the letter~L. But $P$ and $Q$ induce the same $k$-profiles for all $k<11$: the profiles consisting of~$A$ and those $C_i$ or~$C'_i$ that have order~$<k$. These profiles, then, are inequivalent to both $P$ and~$Q$ (witnessed by $Q$ and~$P$, respectively), so both $\sigma$ and~$\tau$ have complexity exactly~11.

As both $P$ and~$Q$ extend our 1-profile~$\{A\}$, the region $\rho$ which $\{A\}$ represents has cohesion at most~10. In fact, it has cohesion exactly~10: for $k=2,\dots,10$ the 1-profile $\{A\}$ extends to exactly the $k$-profiles described above (add all separations of the form~$C_i$ or~$C'_i$ that have order~$<k$), and since these induce each other they are all equivalent to~$\{A\}$. The visibility of the region~$\rho$, thus, is $10-1=9$.\looseness=-1

Now let us determine the cohesion and visibility of the serif region~$\sigma$. Assume that the diagonal line cutting through the serif where it is thickest consists of about 20 edges joining like pairs of pixels. Consider any $k>20$, let $P'$ be a $k$-profile inducing~$P$, and let $D\in P'$ be maximal in the ordering~$\le$ on~$P'$. Then $D$ is $\sub$-minimal among the elements of~$P'$. We claim that $D$ contains only one pixel, showing that $P'$ is focused and hence not an element of the region~$\sigma$. Indeed, if $D$ contains more than one pixel we can partition it into sets $D'$ and~$D''$ which, by our choice of~$k$, both have order~$<k$ as separations. Since $P'$ has the profile property~(P) and contains~$D$, it must also contain one of~$D'$ and~$D''$.%
   \COMMENT{}
   This contradicts the choice of~$D$. Thus, the serif region~$\sigma$ contains no profile of order much bigger than~$20$. Its cohesion, therefore, is at most about~20, and its visibility again (at most) about $20-11=9$.

\medbreak

\begin{figure}[h]
  \center
    \includegraphics[scale=.25]{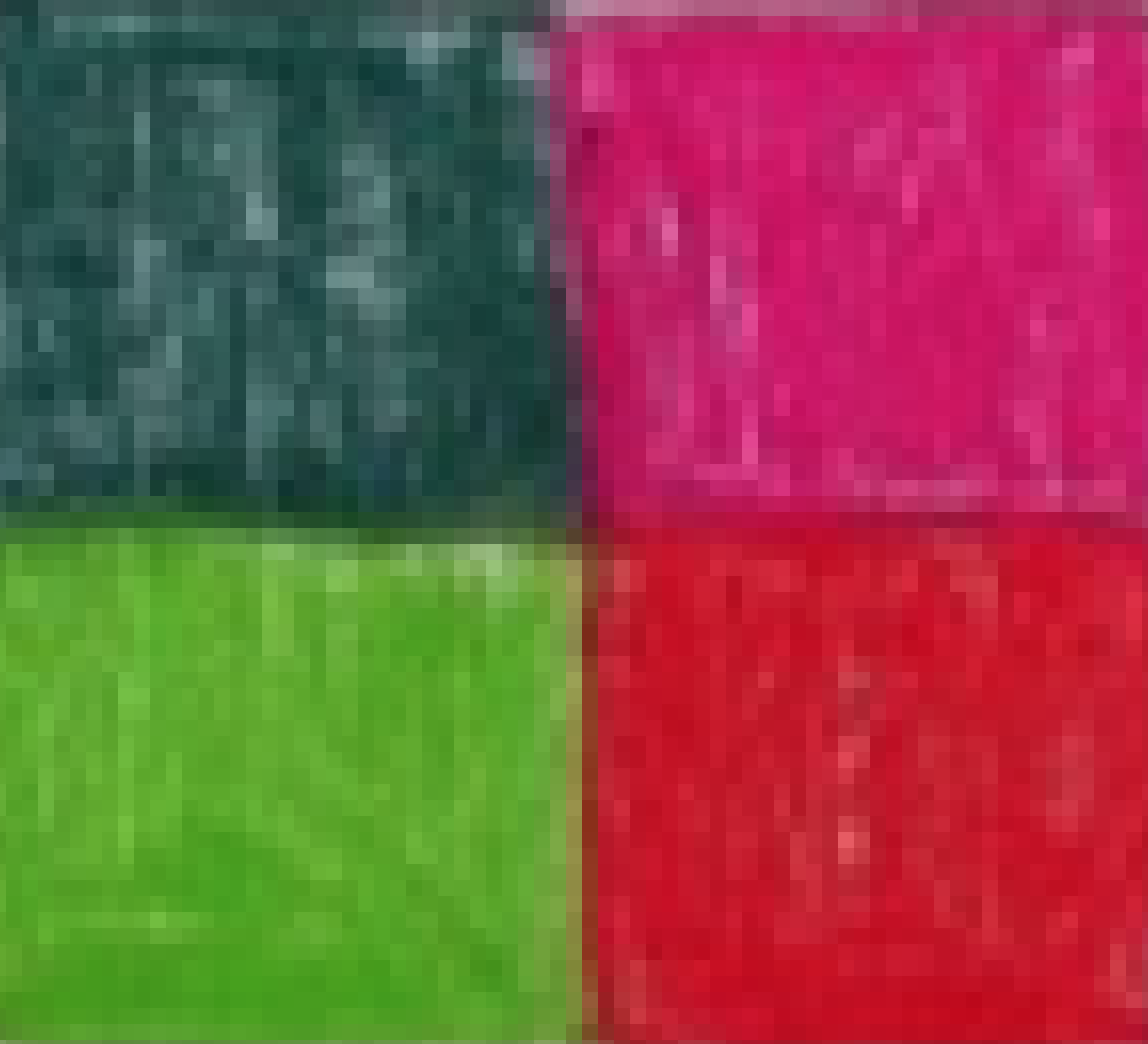}
\caption{Two regions, each refined by two smaller regions}
\label{f2}
\end{figure}

Now consider Figure~\ref{f2}: a square with four quadrants, of which the left two and the right two are similar to each other.

With a natural weighting taking into account the similarity of colours, there is a unique line of lowest order, $k$~say, which runs vertically down the middle. (We ignore all the lines $\{C,C^*\}$ with $C$ a very small set.) This vertical line $\{A,A^*\}$ partitions the picture into a set $A$ of green pixels and a set~$A^*$ of red pixels. There are unique $(k+1)$-profiles~$P$ containing~$\{A\}$ and $Q$ containing~$\{A^*\}$. These represent distinct regions $\gamma$ (for `green') and~$\rho$ (for~`red'). Since $P$ and~$Q$ induce the same $\ell$-profiles for all $\ell \le k$, these regions have complexity~$k+1$ and refine a unique region of cohesion~$k$ and complexity~1.

Let $k'>k$ be the order of the (line around) the bottom green quadrant~$A'$, and $k''>k'$ the order of the top green quadrant~$A''$. (We have $k'' > k'$, because the top two quadrants are more similar to each other than the bottom two quadrants are.) There is a unique $(k'+1)$-profile~$P'$ containing~$A'$, which extends~$P$. The region $\gamma'$ it represents refines~$\gamma$.

This region~$\gamma'$ is distinct from~$\gamma$, because there is another $(k'+1)$-profile extending~$P$ which renders $P'$ inequivalent to~$P$: the unique $(k'+1)$-profile~$P''$ which does {\em not\/} contain~$A'$ but still contains~$A$. This profile thus `points to' the top left quadrant~$A''$. But since $A''$ has order~$k''\ge k'+1$ it does not orient the separation formed by~$A''$ and its complement, so it does not contain~$A''$ as an element.

However, $P''$~is equivalent to the unique $(k''+1)$-profile that does contain~$A''$. The (dark green) region~$\gamma''$ which these two profiles represent thus has complexity~$k'+1$, the same as~$\gamma'$.

Since $P'$ and $P''$ are inequivalent $(k'+1)$-profiles both extending~$P$, the region~$\gamma$ which $P$ represents has cohesion at most~$k'$, and in fact exactly~$k'$. Since the two red quadrants differ in hue less than the two green quadrants do, the region~$\rho$ has a higher cohesion than~$\gamma$. Since all four quadrants have a similar cohesion~-- roughly, the least order of a line cutting right through a quadrant, as in the L example~-- this means that the regions corresponding to the red quadrants are less visible than those of the green quadrants, which bears out our intuition.

The proof of the canonical tangle-tree theorem will produce five nested separations to distinguish these regions: the vertical line separating green from red, and in addition one L-shaped line around each of the four quadrants.

The non-canonical version of the theorem, applied to the set $R$ of the four quadrants (which are pairwise distinguishable,%
   \COMMENT{}
   as the premise of the theorem requires), will then discard two of these latter four separations, retaining just the vertical line and a line around one green and one red quadrant. Note that these three lines still separate all four quadrants, but no subset of just two lines will.\looseness=-1

\begin{figure}[h]
  \center
    \includegraphics[scale=.27]{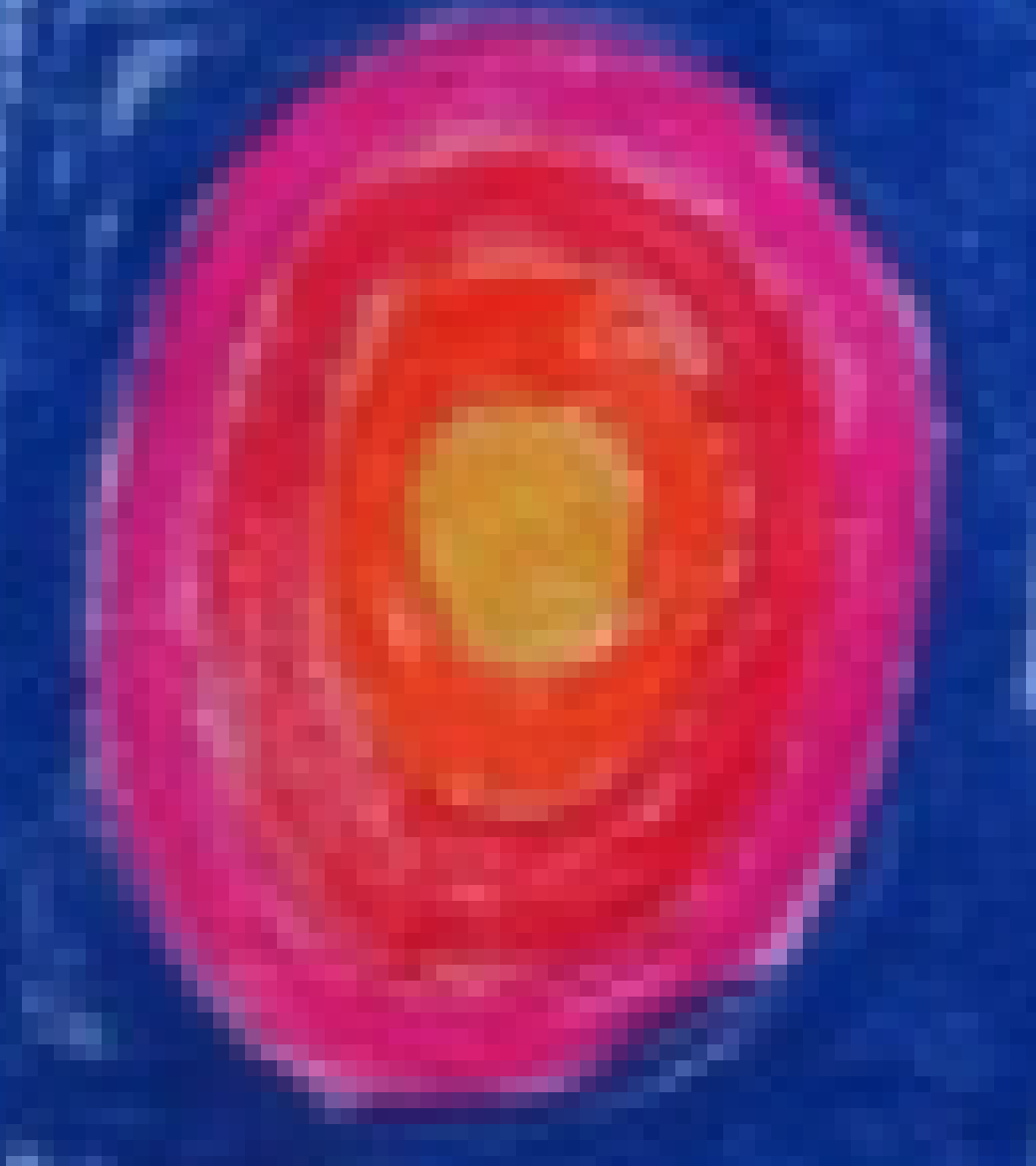}
\caption{Concentric circles have low order, radial lines have high order}
\label{f3}
\end{figure}

\medbreak

Figure~\ref{f3} is intended to illustrate which lines have low order and which do not. The lowest-order separations are the innermost and the outermost circular line between differently coloured regions. The latter has low order, because all its edges~$e$ join a blue pixel to a red one, making $\delta(e)$ large. The innermost circle has smaller values of~$\delta(e)$, but fewer edges in total, making for a similarly low order. The remaining concentric circles mark differences in hue that are about equal in degree, so the longer of these circles have larger order as separations.

The radial lines in Figure~\ref{f3}, by contrast, have maximum values of~$\delta$, since every edge joins two like pixels. Hence the blue background, the yellow innermost disc, and the red concentric bands are the only regions in this picture.

\medbreak

\begin{figure}[h]
  \center
    \includegraphics[scale=.34]{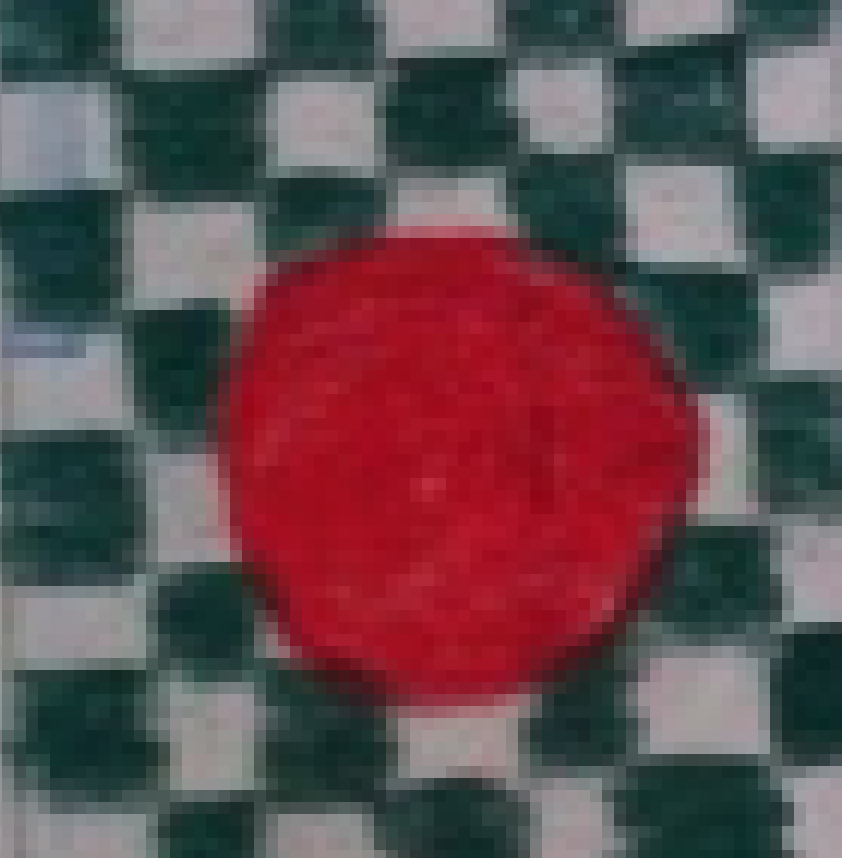}
\caption{Only one region of high cohesion}
\label{f4}
\end{figure}

The inner red disc in picture in Figure~\ref{f4} is a highly visible region, one of low complexity and large cohesion: the order of any straight line that roughly cuts it in half and otherwise runs between differently coloured squares is a lower bound for its cohesion.

The checkerboard background as such does not correspond to a region~$\rho$ in our earlier sense that, for some~$k$, every line of order~$<k$ has most of the backround on one side and the orienting all these lines towards that side defines a profile in~$\rho$. Roughly, the reason for this is that any large portion~$A$ of the background can be partitioned into subsets $A'$ and~$A''$ of at most the same order, by lines running between differently coloured squares. Hence any profile $P$ containing~$A$ must also orient $A'$ and~$A''$, i.e., contain them or their inverses. By the profile property~(P) it cannot contain both inverses as well as~$A$, so it contains at least (and hence, by consistency, exactly) one of~$A'$ and~$A''$. Continuing in this way, we find that $P$ also contains a set~$B$ that is just a small portion of the background rather than most of it, typically one square.

\medbreak

\begin{figure}[h]
  \center
    \includegraphics[scale=.5]{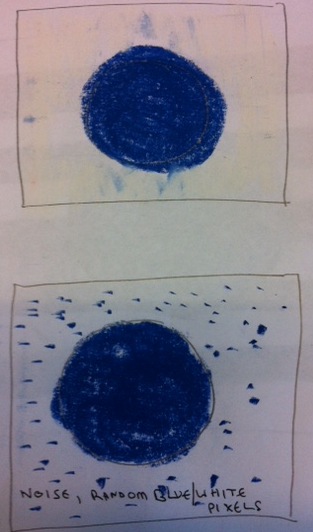}
\caption{Two regions above, only one below}
\label{f6}
\end{figure}

 In the upper image in Figure~\ref{f6}%
   \COMMENT{}
   we see a blue disc against a white background. Here
we have two highly visible regions, of low complexity and high cohesion, separated by the circular line around the blue disc. In the
second image, imagine the background as consisting of blue and white pixels whose colours are chosen independently at random, with equal probability for blue and white. Note that an edge joining two of these random pixels is as likely to join pixels of the same colour as an edge on the circle around the blue disc, which joins a blue pixel to a random pixel.%
   \COMMENT{}
   Hence, as in Figure~\ref{f4}, the central blue disc represents a region of high cohesion, but~-- unlike in the upper image~-- the background does not represent any region at all (with high probability).%
   \COMMENT{}

This changes, however, if the random background is much larger than the blue disc. As soon as $k$ is large enough that $S_k$ contains the line around the blue disc, the orientation of this line outwards is then likely to extend to various $k$-tangles each defined by some uniformly coloured disc just large enough that no line in $S_k$ can cut across it.%
   \COMMENT{}
   If our order function on lines is a little more complicated than in our example,%
   \COMMENT{}
   there will also be a unique tangle corresponding to the entire random background.%
   \COMMENT{}

\section{Outlook}

Our aim in these notes was to show that, in principle, tangles in abstract separation systems can be used to model regions of an image, with the consequence that the tangle-tree theorem and the tangle duality theorem can be used to analyse images. We are aware that our concrete model, as defined in Section~\ref{sec:apps}, amounts to no more than a proof of concept, which can be improved and fine-tuned in many ways.

Our hope is that, nonetheless, the experts in the field may find our approach sufficiently inspiring to be tempted to develop it further.

\begin{figure}[h]
  \center
    \includegraphics[scale=.9]{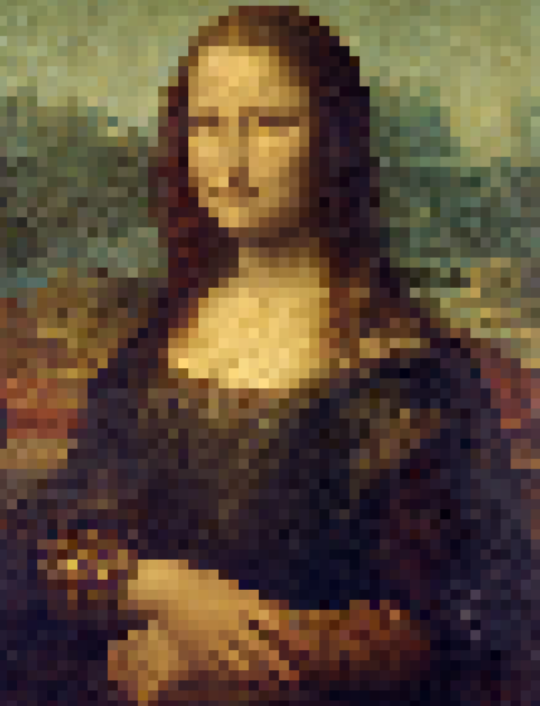}
\caption{Where are the regions of high visibility?}
\label{f7}
\end{figure}

\bibliographystyle{plain}
\bibliography{collective}

\begin{thebibliography}{1}

\bibitem{CDHH13CanonicalAlg}
J.~Carmesin, R.~Diestel, M.~Hamann, and F.~Hundertmark.
\newblock Canonical tree-\penalty-200 decompositions of finite graphs
  {I.~Existence} and algorithms.
\newblock {\em J. Combin. Theory Ser. B}, 116:1--24, 2016.

\bibitem{AbstractSepSys}
R.~Diestel.
\newblock Abstract separation systems.
\newblock To appear in {\em Order} (2017), DOI 10.1007/s11083-017-9424-5,
  arXiv:1406.3797v5.

\bibitem{TreeSets}
R.~Diestel.
\newblock Tree sets.
\newblock To appear in {\em Order} (2017), DOI 10.1007/s11083-017-9425-4,
  arXiv:1512.03781.

\bibitem{ProfileDuality}
R.~Diestel, J.~Erde, and Ph. Eberenz.
\newblock Duality theorem for blocks and tangles in graphs.
\newblock {\em SIAM J.\ Discrete Math.}, 31(3):1514--1528, 2017.

\bibitem{AbstractTangles}
R.~Diestel, J.~Erde, and D.~Wei{\ss}auer.
\newblock Tangles in abstract separation systems.
\newblock In preparation.

\bibitem{ProfilesNew}
R.~Diestel, F.~Hundertmark, and S.~Lemanczyk.
\newblock Profiles of separations: in graphs, matroids, and beyond.
\newblock arXiv:1110.6207, to appear in {\em Combinatorica}.

\bibitem{TangleTreeAbstract}
R.~Diestel and S.~Oum.
\newblock Tangle-tree duality in abstract separation systems.
\newblock arXiv:1701.02509, 2017.

\bibitem{TangleTreeGraphsMatroids}
R.~Diestel and S.~Oum.
\newblock Tangle-tree duality in graphs, matroids and beyond.
\newblock arXiv:1701.02651, 2017.

\bibitem{GMX}
N.~Robertson and P.D. Seymour.
\newblock Graph minors. {X}. {O}bstructions to tree-decomposition.
\newblock {\em J.~Combin.\ Theory (Series B)}, 52:153--190, 1991.

\end{thebibliography}

\small\vfill\noindent Version 31.10.2017

\end{document}